\documentclass[11pt]{article}
\usepackage{latexsym,amsmath,url,caption2,epsfig}
\usepackage{amsfonts,euscript}
\usepackage{graphicx}
\usepackage{amsmath}
\usepackage{amssymb}
\usepackage{amsfonts}
\usepackage{amsthm}
\usepackage{mathrsfs}
\usepackage[all]{xy}
\usepackage{epstopdf}
\usepackage{subfigure}
\usepackage{rotating}
\usepackage{appendix}
\usepackage{url}
\usepackage{setspace}
\usepackage{listings}
\usepackage{color}
\usepackage{tikz}
\usepackage{graphics}
\usepackage[english]{babel}
\usepackage{datetime}
\usepackage{booktabs}
\usepackage{dsfont}
\usepackage{rotfloat}
\usepackage{framed}
\usepackage[utf8]{inputenc}
\usepackage[numbers]{natbib}
\usepackage[breaklinks, colorlinks, citecolor=blue, urlcolor=blue, pdfborder={0 0 1},
citebordercolor={1 1 1}]{hyperref}

\usetikzlibrary{arrows,shapes}
\settimeformat{ampmtime} \numberwithin{equation}{section}
\newcommand{\be}{\begin{equation}}
\newcommand{\ee}{\end{equation}}
\newcommand{\ba}{\begin{aligned}}
\newcommand{\ea}{\end{aligned}}

\newtheorem{theorem}{Theorem}
\newtheorem{corollary}{Corollary}

\newtheorem{lemma}{Lemma}
\newtheorem{remark}{Remark}

\allowdisplaybreaks
\input{tcilatex}
\begin{document}

\title{An Algebra Model for the Higher Order Sum Rules}
\author{Jun Yan}
\maketitle

\begin{abstract}
We introduce an algebra model to study higher order sum rules for orthogonal
polynomials on the unit circle. We build the relation between the algebra
model and sum rules, and prove an equivalent expression on the algebra side
for the sum rules, involving a Hall-Littlewood type polynomial. By this
expression, we recover an earlier result by Golinskii and Zlat\v{o}s, and
prove a new case - half of the Lukic conjecture in the case of a single
critical point with arbitrary order.
\end{abstract}

\section{Introduction}

OPUC (orthogonal polynomials on the unit circle) theory is an important
field in mathematics introduced by Szeg\H{o}, which has not only an
intrinsic interest, but also many applications in different fields such as
Spectral Theory, Random Matrix Theory, Combinatorics and so on (See Chapter
1 in Simon \cite{OPUC1}). It has long been known in the OPUC theory that
there is a one-to-one map between probability measures on the unit circle
and their Verblunsky coefficients. See \cite{OPUC1} for the definition of
Verblunsky coefficients and their related properties. Let $\mathbb{D}$ be
the unit disk. Let $\mu $ be a probability measure on $\partial \mathds{D}$
with infinitely many points in its support, and let $(\alpha _{n})_{n\geq 0}$
be its Verblunsky coefficients. Write%
\begin{equation*}
d\mu (\theta )=w(\theta )\frac{d\theta }{2\pi }+d\mu _{s},
\end{equation*}%
where $d\mu _{s}$ is singular with respect to $d\theta $. Then Szeg\H{o}'s
Theorem (see Simon \cite{OPUC1} and Szeg\H{o} \cite{Sz39}) implies that%
\begin{equation}
\int_{0}^{2\pi }\log \left( w(\theta )\right) \frac{d\theta }{2\pi }>-\infty
\iff \sum_{n=0}^{\infty }\left\vert \alpha _{n}\right\vert ^{2}<\infty \text{%
.}  \label{Szego-result}
\end{equation}%
Results such as (\ref{Szego-result}) are called "spectral theory gems" by
Simon \cite{STID}. There have been lots of studies on higher order sum rules
(gems), that is, to find the necessary and sufficient conditions on the
Verblunsky coefficients side for the event%
\begin{equation}
\int_{0}^{2\pi }\dprod\limits_{j=1}^{K}(1-\cos (\theta -\theta
_{j}))^{m_{j}}\log (w(\theta ))d\theta >-\infty .  \label{Higher-oder-eq}
\end{equation}%
In Simon \cite{OPUC1}, a set of conditions were conjectured; later Lukic (%
\cite{LC1}) found a counterexample to Simon's conditions, and introduced a
modified conjecture. See a discussion of these conditions in \cite{BSZ}. One
thing to note is that Lukic's conditions and Simon's conditions agree when $%
K=1$, and in this case both state that (\ref{Higher-oder-eq}) is equivalent
to ($S$ is the shift operator $(S\alpha )_{n}:=\alpha _{n+1}$)%
\begin{equation}
(S-e^{-i\theta _{1}})^{m_{1}}\alpha \in l^{2}\text{ and }\alpha \in
l^{2m_{1}+2}\text{.}  \label{condition-1}
\end{equation}

A lot of work has been done to study higher order sum rules (\cite{BSZ1},%
\cite{SS06},\cite{GNR2},\cite{GNR3},\cite{GNR1},\cite{GZ07},\cite{LC2},\cite%
{SZ05}). Golinskii and Zlat\v{o}s \cite{GZ07} proved that Simon's conjecture
is correct under the assumption that $\alpha \in l^{4}$. Simon and Zlat\v{o}%
s \cite{SZ05} showed that in the case $(m_{1},m_{2})=(1,1)$ or $(2,0)$,
Simon's conjecture holds. Lukic \cite{LC2} proved that in the case $K=1$,
under the assumption $(S-e^{-i\theta _{1}})\alpha \in l^{2}$, (\ref%
{Higher-oder-eq}) is equivalent to $\alpha \in l^{2m_{1}+2}$. Gamboa, Nagel
and Rouault \cite{GNR1} found the relationship between higher order sum
rules and the large deviations for random matrix models. The large deviation
method was further developped in \cite{BSZ}, where the authors proved the
correctness of Lukic's conjecture in several cases, including the case $%
(m_{1},m_{2})=(2,1)$, where Lukic's conditions differ from Simon's
conditions. The previous work focused on special cases where the orders are
low, or studied the general cases but with additional assumptions. The main
challenge is that all available expressions of (\ref{Higher-oder-eq}) in
terms of the Verblunsky coefficients are difficult to analyze.

In this work, we design an algebra model for the study of higher order sum
rules, and build the relation between them. Under our algebra model, we
prove that the left hand side of (\ref{Higher-oder-eq}) could be expressed
by a Hall-Littlewood type polynomial. With this model, we recover the result
of \cite{GZ07}, and by defining a degree function $L$ which relates Lukic's
conditions to the algebra side, we show that in the case $K=1$ with
arbitrary $m_{1}$, (\ref{condition-1}) leads to (\ref{Higher-oder-eq}).
Moreover, since all Lukic's conditions are related to the polynomial $H$
defined in (\ref{def-H}) below, and that our polynomial (\ref{eq-main-alg})
is also defined by $H$, we expect our representation to shed light on other
cases of Lukic's conjecture as well.

\subsection{Review of Breuer, Simon and Zeitouni \protect\cite{BSZ}}

We begin by stating a theorem, which is a combination of results in \cite%
{BSZ}. Write $\mathds{Z}$ as the set of integers, $\mathds{Z}_{+}$ as the
set of positive integers, and $\mathds{N}$ as the set of non-negative
integers. For $\theta _{1},...,\theta _{K}$ distinct in $[0,2\pi )$, $%
m_{1},...,m_{K}\in $ $\mathds{Z}_{+}$, and $d=\sum_{1\leq j\leq K}m_{j}$,
write%
\begin{equation}
H(e^{i\theta }):=\dprod\limits_{j=1}^{K}(1-\cos (\theta -\theta
_{j}))^{m_{j}}=\frac{1}{2^{d}}\dprod\limits_{j=1}^{K}(e^{i\theta
}-e^{-i\theta _{j}})^{m_{j}}(e^{-i\theta }-e^{i\theta
_{j}})^{m_{j}}=\sum_{l=-d}^{d}h_{l}e^{il\theta },  \label{def-H}
\end{equation}%
where $h_{l}$'$s\in \mathds{C}$. Hereafter we regard $H(\cdot )$ as a
polynomial. Define%
\begin{equation}
Z_{H}:=\frac{1}{2\pi }\int_{0}^{2\pi }H(e^{i\theta })d\theta ,  \label{def-Z}
\end{equation}%
and%
\begin{equation}
V(x):=\frac{-1}{Z_{H}}\left( \sum_{l=1}^{d}\frac{h_{l}}{\left\vert
l\right\vert }x^{l}+\sum_{l=-1}^{-d}\frac{h_{l}}{\left\vert l\right\vert }%
x^{l}\right) .  \label{def-V}
\end{equation}%
Set $\alpha _{n}=0$ if $n<-1$ and $\alpha _{-1}:=-1$. Let $U_{N}$ be the $%
N\times N$ top-left corner of the GGT matrix (Section 4 of \cite{OPUC1}),
that is, $\forall k,l\in \{0,1,...,N-1\},$%
\begin{equation}
(U_{N})_{kl}:=\left\{ 
\begin{array}{cc}
-\alpha _{k-1}\overline{\alpha }_{l}\Pi _{j=k}^{l-1}\rho _{j} & 0\leq k\leq l
\\ 
\rho _{l} & k=l+1 \\ 
0 & k\geq l+2%
\end{array}%
\right. .  \label{def-GGT}
\end{equation}

\begin{theorem}[Breuer, Simon and Zeitouni \protect\cite{BSZ}]
\label{theorem-BSZ}The inequality (\ref{Higher-oder-eq}) holds if and only if%
\begin{equation}
\limsup_{N\rightarrow \infty }\left( \text{Tr}(V(U_{N}))-\sum_{n=0}^{N-1}%
\log (1-\left\vert \alpha _{n}\right\vert ^{2})\right) <\infty .
\label{eq-limit-1}
\end{equation}
\end{theorem}

From Theorem 3.2 in \cite{BSZ}, we can write%
\begin{equation*}
\text{Tr}(V(U_{N}))=bdy+\sum_{j=0}^{N-1-d}G(\alpha _{j},...,\alpha _{j+d}),
\end{equation*}%
where $G$ is a degree $2d$ polynomial, and $bdy$ stands for boundary terms
whose absolute value is bounded by an $N$-independent constant $C$. So the
main focus is the part $\sum_{j=0}^{N-1-d}G(\alpha _{j},...,\alpha _{j+d})$.
Note that $G(\alpha _{j},...,\alpha _{j+d})$ is not unique as mentioned in 
\cite{BSZ}, because we can do shifts on the indices, such as replacing $%
\alpha _{j+1}\overline{\alpha _{j+2}}$ by $\alpha _{j}\overline{\alpha _{j+1}%
}$. In \cite{BSZ}, a particular $G$ is calculated in several simple cases.
Note that any choice of $G$ consists of even degree terms (see the remark
above Theorem 3.3 in \cite{BSZ}), which allows us to write $%
G=\sum_{k=1}^{d}G_{2k}$, where $G_{2k}$ is a degree $2k$ homogeneous
polynomial. In the next subsection, we introduce an algebra for studying $G$
and provide a corresponding expression for it.

\subsection{An algebra model for gems\label{sub-section-alg}}

For $k\in \mathds{Z}_{+}$, we consider the polynomial ring $A_{2k}:=%
\mathds{C}\lbrack x_{1},y_{1},...,x_{k},y_{k}]$. Given $(\alpha _{n})_{n\geq
0}\in \mathbb{D}^{\infty }$, we define a linear map $\phi _{2k}:$ $%
A_{2k}\rightarrow \mathbb{D}^{\infty }$, such that%
\begin{equation*}
\lbrack \phi _{2k}(\dprod\limits_{i=1}^{k}x_{i}^{\beta _{i}}y_{i}^{\gamma
_{i}})]_{n}=\dprod\limits_{i=1}^{k}\alpha _{n+\beta _{i}}\overline{\alpha
_{n+\gamma _{i}}}\text{ for all }\beta _{i},\gamma _{i}\in \mathds{N},i\in
\lbrack k].
\end{equation*}%
Let $\widetilde{A}_{2k}:=\mathds{C}\lbrack x_{1},y_{1},...,x_{k},y_{k},\frac{%
1}{x_{1}},\frac{1}{y_{1}},...,\frac{1}{x_{k}},\frac{1}{y_{k}}]$. Define the
factor rings $B_{2k}:=A_{2k}/(\prod_{i=1}^{k}x_{i}y_{i}-1)$, $\widetilde{B}%
_{2k}:=\widetilde{A}_{2k}/(\prod_{i=1}^{k}x_{i}y_{i}-1)$, and write $\psi
_{2k}$ ($\widetilde{\psi }_{2k}$) as the natural homomorphism from $A_{2k}$ (%
$\widetilde{A}_{2k}$) to $B_{2k}$ ($\widetilde{B}_{2k}$). Then we have

\begin{lemma}
\label{lemma-alg-1}For each $k\in \mathds{Z}_{+}$, if two polynomials $%
G^{(1)}$ and $G^{(2)}\in A_{2k}$ have the same image under $\psi _{2k}$,
then there exists $C<\infty $ (might depend on $G^{(1)}$ and $G^{(2)}$, but
is $N$-independent) such that for any $N\in \mathds{Z}_{+}$, we have%
\begin{equation*}
\left\vert \sum_{n=0}^{N}\left[ \phi _{2k}(G^{(1)})\right]
_{n}-\sum_{n=0}^{N}\left[ \phi _{2k}(G^{(2)})\right] _{n}\right\vert <C\text{%
.}
\end{equation*}
\end{lemma}

The motivation to consider the factor ring $B_{2k}$ is that it describes the
fact that we can do indices shifts for $G$ when considering gems. For
example, as mentioned in the previous subsection, one can replace $\alpha
_{j+1}\overline{\alpha _{j+2}}$ by $\alpha _{j}\overline{\alpha _{j+1}}$ in $%
G(\alpha _{j},...,\alpha _{j+d})$. Note that the preimages of $\alpha _{j+1}%
\overline{\alpha _{j+2}}$ and $\alpha _{j}\overline{\alpha _{j+1}}$ in $%
A_{2} $ are $x_{1}^{j+1}y_{1}^{j+2}$ and $x_{1}^{j}y_{1}^{j+1}$
respectively, which have the same image under $\psi _{2}$. The next theorem
provides us with an alternative way to show that $G^{(1)}$, $G^{(2)}$ have
the same image under $\psi _{2k}$.

\begin{lemma}
\label{lemma-alg-2}For each $k\in \mathds{Z}_{+}$, if $G^{(1)}$, $G^{(2)}\in
A_{2k}$ have the same image under $\widetilde{\psi }_{2k}$, then they also
have the same image under $\psi _{2k}$.
\end{lemma}

Let $a_{k,p}:=\prod_{s=p}^{k}y_{s}\prod_{s=p+1}^{k}x_{s},$ $%
b_{k,p}:=\prod_{s=1}^{p}x_{s}y_{s}$ for $p\in \lbrack k]$. Following is the
main theorem of our algebra model.

\begin{theorem}
\label{Prop-G-prime}The polynomial $G_{2k}^{\prime }\in \widetilde{A}_{2k}$
and $G_{2k}$ has the same image under $\widetilde{\psi }_{2k}$, where%
\begin{equation}
G_{2k}^{\prime }:=\frac{(-1)^{k+1}}{kZ_{H}}\sum_{1\leq p,q\leq k}\frac{%
H\left( a_{k,p}b_{k,q}\right) }{\dprod\limits_{s\in \lbrack k]\backslash
\{p\}}(1-\frac{a_{k,s}}{a_{k,p}})\dprod\limits_{t\in \lbrack k]\backslash
\{q\}}(\frac{b_{k,q}}{b_{k,t}}-1)}-\frac{1}{k}.  \label{eq-main-alg}
\end{equation}
\end{theorem}

\begin{remark}
While each element in the summation of $G_{2k}^{\prime }$ might not be in $%
\widetilde{A}_{2k}$, $G_{2k}^{\prime }$ is in $\widetilde{A}_{2k}$. It is a
Hall-Littlewood type polynomial (see Section 3.2 in \cite{SFHP}).
\end{remark}

\begin{remark}
\label{remark-1}Note that $\phi _{2k}$ is not a one-to-one map. For example, 
$x_{1}y_{1}x_{2}^{2}y_{2}^{2}$ has the same image as $%
x_{1}^{2}y_{1}^{2}x_{2}y_{2}$. Indeed, for any polynomial, one can apply
permutations on $\{x_{1},...,x_{k}\}$ and $\{y_{1},...,y_{k}\}$, without
changing its image under $\phi _{2k}$. However, $a_{k,p}$ and $b_{k,p}$ are
defined for a fixed order of $x_{i}$'$s$ and $y_{i}$'$s$. Therefore one can
get a set of polynomials having the same image with $G_{2k}$ under $%
\widetilde{\psi }_{2k}$, by changing the order of $\{x_{1},...,x_{k}\}$ and $%
\{y_{1},...,y_{k}\}$, and defining $a_{k,p}$, $b_{k,p}$ and $G_{2k}^{\prime
} $ correspondingly. One can also do averaging of these polynomials to get a
bipartite symmetric polynomial in $\{x_{1},...,x_{k}\}$ and $%
\{y_{1},...,y_{k}\}$.
\end{remark}

\begin{remark}
Observing that Lukic's conditions are all related to $H$, an advantage of (%
\ref{eq-main-alg}) is that it directly relates the left hand side of (\ref%
{Higher-oder-eq}) to the polynomial $H$. Thus we expect that for the general
cases of the Lukic conjecture, after a deeper analysis of (\ref{eq-main-alg}%
), or other polynomials mentioned in Remark \ref{remark-1}, one can find a
way to decompose the polynomial such that each component in the
decomposition is controlled by Lukic's conditions.
\end{remark}

With Theorem \ref{Prop-G-prime} we get the following corollary.

\begin{corollary}
\label{corollary}The inequality (\ref{Higher-oder-eq}) holds if and only if the
following limit superior $<\infty $:{\small 
\begin{equation}
\limsup_{N\rightarrow \infty }\sum_{n=0}^{N-1}\big(\sum_{k=1}^{d}\big[\phi
_{2k}(\frac{(-1)^{k+1}}{kZ_{H}}\sum_{1\leq p,q\leq k}\frac{H\left(
a_{k,p}b_{k,q}\right) (\prod_{i\in \lbrack k]}x_{i}y_{i})^{2k}}{%
\dprod\limits_{s\in \lbrack k]\backslash \{p\}}(1-\frac{a_{k,s}}{a_{k,p}}%
)\dprod\limits_{t\in \lbrack k]\backslash \{q\}}(\frac{b_{k,q}}{b_{k,t}}-1)})%
\big]_{n}-\log (1-\left\vert \alpha _{n}\right\vert ^{2})-\sum_{k=1}^{d}%
\frac{\left\vert \alpha _{n}\right\vert ^{2k}}{k}\big).
\end{equation}%
}
\end{corollary}

\begin{remark}
The additional term $(\prod_{i\in \lbrack k]}x_{i}y_{i})^{2k}$ is to make
sure that the polynomial inside $\phi _{2k}$ is in $A_{2k}$.
\end{remark}

\begin{remark}
From this corollary we see that the term $-1/k$ in (\ref{eq-main-alg}) is
not redundant but rather surprising, since $\left[ \phi _{2k}(-1/k)\right]
_{n}$ perfectly matches the $k$th order expansion of $\log (1-\left\vert
\alpha _{n}\right\vert ^{2})$. Thus when $\left\Vert \alpha \right\Vert
_{d}<\infty $ (for example under Lukic's conditions), these terms $-1/k$ for 
$k\in \lbrack d]$ perfectly cancel the $\log $ terms in (\ref{eq-limit-1})
up to a constant.
\end{remark}

\subsection{Higher order sum rules}

With the above theorems, we can recover the following theorem.

\begin{theorem}[Golinskii and Zlat\v{o}s \protect\cite{GZ07}]
\label{theorem-GZ}If $\alpha \in l^{4}$, then (\ref{Higher-oder-eq}) is
equivalent to%
\begin{equation}
\dprod\limits_{j=1}^{K}(S-e^{-i\theta _{j}})^{m_{j}}\alpha \in l^{2}.
\label{condition-gz}
\end{equation}
\end{theorem}

Then we prove that

\begin{theorem}
\label{Theorem-n-0}In the case that $K=1$ and arbitrary $m_{1}$, (\ref%
{condition-1}) implies (\ref{Higher-oder-eq}).
\end{theorem}

Theorem \ref{Theorem-n-0} is a new result in higher order sum rules. Note
that it is similar to Lukic \cite{LC2} but different, since in \cite{LC2}
there is an assumption $(S-e^{-i\theta _{1}})\alpha \in l^{2}$ which is
stronger than (\ref{condition-1}).

As we can see, the analysis focuses on three parts: the Verblunsky
coefficients part, the algebra part, and the sum rules part. In the rest of
this paper we provide the proofs in the three parts separately. In Section %
\ref{proof-vc-part}, we show how to derive Theorem \ref{theorem-BSZ} based
on \cite{BSZ}, and we provide Lemma \ref{lemma-combinatorics}, which is a
key step for Theorem \ref{Prop-G-prime}. In Section \ref{proof-alg-part}, we
give the proofs for all statements of Subsection \ref{sub-section-alg}. In
Section \ref{proof-sumrules-part}, we show how to get Theorem \ref%
{theorem-GZ}, and prove Theorem \ref{Theorem-n-0} by the discrete
Gagliardo-Nirenberg inequality and defining a degree function $L_{2k}$.

\section{Proofs of the Verblunsky coefficients part\label{proof-vc-part}}

In this section, first we show how to combine the results of \cite{BSZ} to
get Theorem \ref{theorem-BSZ}, and then provide Lemma \ref%
{lemma-combinatorics}, which is an important step for Theorem \ref%
{Prop-G-prime}.

\begin{proof}[Proof of Theorem \protect\ref{theorem-BSZ}]
Substituting equations \cite[(3.6) and (3.7)]{BSZ} into \cite[(3.4)]{BSZ},
after simple calculations we can verify that the function $V(\cdot )$
defined in (\ref{def-V}) is the same as the $V(\cdot )$ in equation \cite[%
(3.4)]{BSZ} with $C=0$. With equations \cite[(1.3) and (3.6)]{BSZ}, we see
that $H(\eta \mid \mu )$ is the integral we care about. With \cite[Theorem
3.2 and Theorem 3.5]{BSZ}, we see that Theorem \ref{theorem-BSZ} holds if $%
U_{N}$ is an $N\times N$ unitary CMV matrix. With the discussions above
equation \cite[(9.14)]{BSZ}, it is easy to see that it also holds for $U_{N}$
the $N\times N$ top-left corner of the GGT matrix.
\end{proof}

Next, we prove the following lemma, which analyzes the degree $2k$ terms in
Tr$((U_{N})^{l})$. Write%
\begin{equation*}
\widetilde{D}_{2k,l}:=\{(i_{1},j_{1},...,i_{k},j_{k}):\text{ }%
\sum_{s=1}^{k}\left( j_{s}-i_{s}\right) =l,\text{ }\forall s\in \lbrack
k],i_{s},j_{s}\in \mathds{N},j_{s}\geq i_{s},j_{s}>i_{s+1},i_{k+1}:=i_{1}\},
\end{equation*}%
and%
\begin{equation*}
D_{2k,l}:=\widetilde{D}_{2k,l}\cap \{(i_{1},j_{1},...,i_{k},j_{k}):i_{1}=0\}.
\end{equation*}%
For $F$ a polynomial in $\mathds{C}\lbrack \alpha _{n},n\geq -1]$, denote by 
$g_{2k}(F)$ the degree $2k$ terms in $F$. Then we have

\begin{lemma}
\label{lemma-combinatorics}For each $l\in \mathds{Z}_{+}$, there exists $%
C_{l}<\infty $ such that $\forall N\in \mathbb{Z}_{+}$,%
\begin{equation*}
\left\vert \ g_{2k}\left( \text{Tr}((U_{N})^{l})\right) -(-1)^{k}\frac{l}{k}%
\sum_{n=0}^{N}\sum_{D_{2k,l}}\dprod\limits_{p=1}^{k}\alpha _{n+i_{p}}%
\overline{\alpha _{n+j_{p}}}\right\vert <C_{l}.
\end{equation*}%
Similarly, for $l\in \mathds{Z}_{-}$, there exists $C_{l}<\infty $ such that 
$\forall N\in \mathbb{Z}_{+}$,%
\begin{equation*}
\left\vert g_{2k}\left( \text{Tr}((U_{N})^{l})\right) -(-1)^{k}\frac{%
\left\vert l\right\vert }{k}\sum_{n=0}^{N}\sum_{D_{2k,l}}\dprod%
\limits_{p=1}^{k}\overline{\alpha _{n+i_{p}}}\alpha _{n+j_{p}}\right\vert
<C_{l}.
\end{equation*}
\end{lemma}

\begin{proof}[Proof of Lemma \protect\ref{lemma-combinatorics}]
Fix $l\in \mathds{Z}_{+}$, and write $\rho _{l}:=\sqrt{1-\left\vert \alpha
_{l}\right\vert ^{2}}$. Define an $N\times N$ matrix $\widetilde{U}_{N}$ as
follows: $\forall k,l\in \{0,1,...,N-1\},$%
\begin{equation*}
(\widetilde{U}_{N})_{kl}:=\left\{ 
\begin{array}{cc}
-\alpha _{k-1}\overline{\alpha }_{l} & 0\leq k\leq l \\ 
\rho _{l}^{2} & k=l+1 \\ 
0 & k\geq l+2%
\end{array}%
\right. .
\end{equation*}%
We claim that Tr$((U_{N})^{l})=$Tr$((\widetilde{U}_{N})^{l})$. To see this,
define%
\begin{equation*}
F_{l}:=\{(i_{1},...,i_{l}):\forall s\in \lbrack l],\text{ }i_{s}\in
\{0,1,...,N-1\},i_{s+1}\geq i_{s}-1,i_{l+1}:=i_{1}\},
\end{equation*}%
then Tr$((U_{N})^{l})=\sum_{(i_{1},...,i_{l})\in
F_{l}}(U_{N})_{i_{1}i_{2}}(U_{N})_{i_{2}i_{3}}...(U_{N})_{i_{l}i_{l+1}}$,
and the same equation holds for $\widetilde{U}_{N}$. For each $%
(i_{1},...,i_{l})\in F_{l}$, by the form of the GGT matrix in (\ref{def-GGT}%
) we observe that%
\begin{equation*}
(U_{N})_{i_{1}i_{2}}(U_{N})_{i_{2}i_{3}}...(U_{N})_{i_{l}i_{l+1}}=\dprod%
\limits_{s=1}^{l}\left( 1_{\{i_{s+1}=i_{s}-1\}}\rho
_{i_{s+1}}+1_{\{i_{s+1}>i_{s}-1\}}(-\alpha _{i_{s}-1}\overline{\alpha }%
_{i_{s+1}})\dprod\limits_{q=i_{s}}^{i_{s+1}-1}\rho _{q}\right) .
\end{equation*}%
Assume $i_{n_{1}}<i_{n_{2}}<...<i_{n_{p}}$ such that $\{i_{n_{t}},t\in
\lbrack p]\}=\{s:s\in \lbrack l],i_{s+1}>i_{s}-1\}$. Let $n_{p+1}=n_{1}$,
and define $N_{q}^{(1)}:=\#\{t\in \lbrack p]:i_{n_{t}}\leq q\leq
i_{n_{t}+1}-1\}$, $N_{q}^{(2)}:=\#\{t\in \lbrack p]:i_{n_{t+1}}\leq q\leq
i_{n_{t}+1}-1\}$. We claim that $N_{q}^{(1)}=N_{q}^{(2)}$. This is because,
if we draw a graph such that $f(2t-1)=i_{n_{t}}-1/2$ for $t\in \lbrack p+1]$%
, $f(2t)=i_{n_{t}+1}-1/2$ for $t\in \lbrack p]$, and connect adjacent pairs
of points by lines, then $N_{q}^{(1)},N_{q}^{(2)}$ are respectively the
numbers of upcrossings and downcrossings of $f$ w.r.t. the level $p$, which
must be equal by the fact that $f(2p+1)=f(1)$. From this we see that%
\begin{equation*}
\dprod\limits_{t=1}^{p}\left(
\dprod\limits_{q=i_{n_{t}}}^{i_{n_{t}+1}-1}\rho _{q}\right)
=\dprod\limits_{q}\rho _{q}^{N_{q}^{(1)}}=\dprod\limits_{q}\rho
_{q}^{N_{q}^{(2)}}=\dprod\limits_{t=1}^{p}\dprod%
\limits_{q=i_{n_{t+1}}}^{i_{n_{t}+1}-1}\rho _{q},
\end{equation*}%
which implies that%
\begin{eqnarray*}
(U_{N})_{i_{1}i_{2}}(U_{N})_{i_{2}i_{3}}...(U_{N})_{i_{l}i_{l+1}}
&=&\dprod\limits_{t=1}^{p}\left( \left( -\alpha _{i_{n_{t}}-1}\overline{%
\alpha }_{i_{n_{t}+1}}\right)
\dprod\limits_{q=i_{n_{t}}}^{i_{n_{t}+1}-1}\rho
_{q}\dprod\limits_{q=i_{n_{t+1}}}^{i_{n_{t}+1}-1}\rho _{q}\right)  \\
&=&\dprod\limits_{t=1}^{p}\left( \left( -\alpha _{i_{n_{t}}-1}\overline{%
\alpha }_{i_{n_{t}+1}}\right)
\dprod\limits_{q=i_{n_{t+1}}}^{i_{n_{t}+1}-1}\rho _{q}^{2}\right) =(%
\widetilde{U}_{N})_{i_{1}i_{2}}(\widetilde{U}_{N})_{i_{2}i_{3}}...(%
\widetilde{U}_{N})_{i_{l}i_{l+1}}.
\end{eqnarray*}%
Therefore Tr$((U_{N})^{l})=$Tr$((\widetilde{U}_{N})^{l})$. For $k\in %
\mathds{Z}_{+}$, we consider the terms in $g_{2k}\left( \text{Tr}%
((U_{N})^{l})\right) $. Each term has the form $%
u_{i_{1}i_{2}}u_{i_{2}i_{3}}...u_{i_{l}i_{l+1}}$, where $(i_{1},...,i_{l})%
\in F_{l}$, $u_{i_{s}i_{s+1}}=-\alpha _{i_{s+1}}\overline{\alpha }_{i_{s+1}}$
or $1$ if $i_{s+1}=i_{s}-1$, $u_{i_{s}i_{s+1}}=-\alpha _{i_{s}-1}\overline{%
\alpha }_{i_{s+1}}$ if $i_{s+1}>i_{s}-1$, and $\#\{s\in \lbrack
l]:u_{i_{s}i_{s+1}}\neq 1\}=k$. Let%
\begin{equation*}
\{n_{1},n_{2},...,n_{k}\}=\{s\in \lbrack l]:u_{i_{s}i_{s+1}}\neq 1\}\text{, }%
n_{1}<n_{2}<...<n_{k}.
\end{equation*}%
We define the map $\varphi $ from $\widetilde{D}_{2k,l}$ to a degree $2k$
monomial, such that $\varphi
((i_{1},j_{1},...,i_{k},j_{k}))=\prod_{p=1}^{k}\alpha _{i_{p}}\overline{%
\alpha _{j_{p}}}$. Write%
\begin{equation*}
\Lambda
:=\{(1,2,...,k),(2,3,...,k,1),(3,4,...,k,1,2),...,\{k,1,2,...,k-1\}\}.
\end{equation*}%
Consider the following weight distributing operation: assume that each $%
u_{i_{1}i_{2}}u_{i_{2}i_{3}}...u_{i_{l}i_{l+1}}$ has weight $1$, and
uniformly distributes its weight to $k$ objects: $(i_{n_{\pi
(1)}}-1,i_{n_{\pi (1)}+1},i_{n_{\pi (2)}}-1,i_{n_{\pi (2)}+1},...,i_{n_{\pi
(k)}}-1,i_{n_{\pi (k)}+1})$ for $\pi \in \Lambda $. This operation
corresponds to the following identity%
\begin{equation}
u_{i_{1}i_{2}}u_{i_{2}i_{3}}...u_{i_{l}i_{l+1}}=\frac{1}{k}\sum_{\pi \in
\Lambda }\varphi ((i_{n_{\pi (1)}}-1,i_{n_{\pi (1)}+1},i_{n_{\pi
(2)}}-1,i_{n_{\pi (2)}+1},...,i_{n_{\pi (k)}}-1,i_{n_{\pi (k)}+1})).
\label{eq-weight-distribution}
\end{equation}%
It is easy to verify that $\forall \pi \in \Lambda $, $(i_{n_{\pi
(1)}}-1,i_{n_{\pi (1)}+1},i_{n_{\pi (2)}}-1,i_{n_{\pi (2)}+1},...,i_{n_{\pi
(k)}}-1,i_{n_{\pi (k)}+1})\in \widetilde{D}_{2k,l}$. Conversely, we claim
that for each $n\in \lbrack 2d,N-2d]$ and $(i_{1},j_{1},...,i_{k},j_{k})\in $
$D_{2k,l}$, the term $(n+i_{1},n+j_{1},...,n+i_{k},n+j_{k})$ receives weight 
$l/k$. To see this, clockwisely choose $l$ positions on a circle. Put $%
-\alpha _{n+i_{1}}\overline{\alpha _{n+j_{1}}}$ at the $1$st position, and
put $j_{1}-i_{2}-1$ number of $1$'$s$ in the next positions, then put $%
-\alpha _{n+i_{2}}\overline{\alpha _{n+j_{2}}}$ at the $\left(
j_{1}-i_{2}\right) $th position. Continue doing this until all the $l$
positions are filled. Then we can observe that, each preimage of $%
(n+i_{1},n+j_{1},...,n+i_{k},n+j_{k})$ corresponds to a length-$l$
consecutive sequence on the circle, which has $l$ choices depending on how
to choose the starting position. For $n\in \lbrack 2d,N-2d]$, since the $%
l_{\infty }$ norm of the elements in $D_{2k,l}$ is bounded by $l\leq d$, it
is easy to verify that all the $l$ choices correspond to $l$ terms in $%
g_{2k}\left( \text{Tr}((U_{N})^{l})\right) $. Thus $%
(n+i_{1},n+j_{1},...,n+i_{k},n+j_{k})$ receives weight $l/k$ (in some cases,
due to symmetry these $l$ choices of sequences are not all different. For
example, it is possible that there are just $M$ different choices with $M$
divides $l$. However this makes no influence: in this case, when we start
from each one of these $M$ choices and distribute their weights, there are
just $M$ different elements in $D_{2k,l}$ receiving weights, with each of
them getting $M/k$ weights. Thus we can easily verify that each element in $%
D_{2k,l}$ receives $l/k$ weights). Interpreting the weights as the
coefficients in the summation as (\ref{eq-weight-distribution}), and noting
that the contribution of other terms with $i_{1}\notin \lbrack 2d,N-2d]$ is
controlled by a constant independent of $N$, we complete the proof for $l\in %
\mathds{Z}_{+}$. Similar applies to $l\in \mathds{Z}_{-}$.
\end{proof}

\section{Proofs of the algebra part\label{proof-alg-part}}

In this section we prove all the statements in Section \ref{sub-section-alg}%
. First we show Lemma \ref{lemma-alg-1}.

\begin{proof}[Proof of Lemma \protect\ref{lemma-alg-1}]
Since $\phi _{2k}$ is linear, it is enough to show that if $G^{(3)}\in A_{2k}
$ and $\psi _{2k}(G^{(3)})=0$ in $B_{2k}$, then there exists an $N$%
-independent $C<\infty $ such that $\left\vert \sum_{n=0}^{N}\left[ \phi
_{2k}(G^{(3)})\right] _{n}\right\vert <C$ for all $N$. Since $\psi
_{2k}(G^{(3)})=0$, we can write $G^{(3)}=(\prod_{i=1}^{k}x_{i}y_{i}-1)\left(
\sum_{s=1}^{M}c_{s}\prod_{i=1}^{k}x_{i}^{p_{i,s}}y_{i}^{q_{i,s}}\right) $
where $c_{s}\neq 0,p_{i,s},q_{i,s}\in \mathds{N}$ for $s\in \lbrack M],i\in
\lbrack k]$. By the linearity of $\phi _{2k}$ it suffices to show that $%
\forall s\in \lbrack M]$ there exists an $N$-independent $C_{s}<\infty $
such that $|\sum_{n=0}^{N}[\phi
_{2k}((\prod_{i=1}^{k}x_{i}y_{i}-1)c_{s}%
\prod_{i=1}^{k}x_{i}^{p_{i,s}}y_{i}^{q_{i,s}})]_{n}|<C_{s}$ for all $N$,
which is implied by the fact that $\forall N\in \mathds{Z}_{+}$%
\begin{equation*}
\left\vert \sum_{n=0}^{N}[\phi
_{2k}((\dprod\limits_{i=1}^{k}x_{i}y_{i}-1)c_{s}\dprod%
\limits_{i=1}^{k}x_{i}^{p_{i,s}}y_{i}^{q_{i,s}})]_{n}\right\vert
=c_{s}\left\vert \dprod\limits_{i=1}^{k}\alpha _{N+1+\beta _{i}}\overline{%
\alpha _{N+1+\gamma _{i}}}-\dprod\limits_{i=1}^{k}\alpha _{\beta _{i}}%
\overline{\alpha _{\gamma _{i}}}\right\vert \leq 2c_{s}\text{,}
\end{equation*}%
where in the rightmost inequality we use the fact that $\left\Vert \alpha
\right\Vert _{\infty }\leq 1$.
\end{proof}

Next we show Lemma \ref{lemma-alg-2}.

\begin{proof}[Proof of Lemma \protect\ref{lemma-alg-2}]
Since $\phi _{2k}$ is linear, it is enough to show that if $G^{(3)}\in A_{2k}
$ and $\widetilde{\psi }_{2k}(G^{(3)})=0$ in $\widetilde{B}_{2k}$, then $%
\psi _{2k}(G^{(3)})=0$. We can write $G^{(3)}=(\prod_{i=1}^{k}x_{i}y_{i}-1)(%
\sum_{s=1}^{M}c_{s}\prod_{i=1}^{k}x_{i}^{p_{i,s}}y_{i}^{q_{i,s}})$ where $%
c_{s}\neq 0,p_{i,s},q_{i,s}\in \mathds{Z}$ for $s\in \lbrack M],i\in \lbrack
k]$, and $(p_{i,s},q_{i,s})_{1\leq i\leq k}\,$are different for different $s$%
. We claim that we must have $p_{i,s},q_{i,s}\in \mathds{N}$ for $s\in
\lbrack M],i\in \lbrack k]$. Otherwise, without loss of generality we assume 
$p_{1,1}<0$. Consider the set $\Gamma :=\{s_{0}\in \lbrack M]:s_{0}=\arg
\min \{p_{1,s}\}\}$. Since $(p_{i,s},q_{i,s})_{1\leq i\leq k}\,$are
different for $s\in \Gamma $, $\sum_{s\in \Gamma
}c_{s}y_{i}^{q_{1,s}}\prod_{i=2}^{k}x_{i}^{p_{i,s}}y_{i}^{q_{i,s}}\neq 0$.
Therefore after expanding $G^{(3)}$ according to the degree of $x_{1}$ as an
element in $\mathds{C}\lbrack y_{1},...,x_{k},y_{k},\frac{1}{y_{1}},...,%
\frac{1}{x_{k}},\frac{1}{y_{k}}][x_{1},\frac{1}{x_{1}}]$, there is a term $%
-x_{1}^{\min \{p_{1,s}\}}\sum_{s\in \Gamma
}c_{s}y_{i}^{q_{1,s}}\dprod\limits_{i=2}^{k}x_{i}^{p_{i,s}}y_{i}^{q_{i,s}}$.
Since $\min \{p_{i,s}\}<0$, we see that $G^{(3)}\notin A_{2k}$, leading to a
contradiction. So $p_{i,s},q_{i,s}\in \mathds{N}$ for all $s\in \lbrack
M],i\in \lbrack k]$, and it completes the proof.
\end{proof}

Based on Lemma \ref{lemma-combinatorics}, we prove Theorem \ref{Prop-G-prime}
in the following.

\begin{proof}[Proof of Theorem \protect\ref{Prop-G-prime}]
For each $l\in \mathds{Z}_{+}$, note that $[\phi
_{2k}(\sum_{D_{2k,l}}\prod_{p=1}^{k}x_{_{p}}^{i_{p}}y_{_{p}}^{j_{p}})]_{n}=%
\sum_{D_{2k,l}}\prod_{p=1}^{k}\alpha _{n+i_{p}}\overline{\alpha _{n+j_{p}}}$%
. With Lemma \ref{lemma-combinatorics} it suffices to calculate $%
\sum_{D_{2k,l}}\prod_{p=1}^{k}x_{_{p}}^{i_{p}}y_{_{p}}^{j_{p}}$. Write%
\begin{equation*}
E_{k,l}:=\{(v_{1},...,v_{k}):\forall p\in \lbrack k],v_{p}\in \mathds{N}%
,\sum_{p=1}^{k}v_{p}=l\},
\end{equation*}%
\begin{equation*}
\widetilde{E}_{k,l}:=\{(\widetilde{v}_{1},...,\widetilde{v}_{k}):\forall
p\in \lbrack k],\widetilde{v}_{p}\in \mathds{Z}_{+},\sum_{p=1}^{k}\widetilde{%
v}_{p}=l\}.
\end{equation*}%
It is not hard to see that there is a one-to-one map between $E_{k,l}\times 
\widetilde{E}_{k,l}$ and $D_{2k,l}$ as follows:%
\begin{equation}
i_{p}=\sum_{s=1}^{p-1}\left( v_{s}-\widetilde{v}_{s}\right) ,\text{ }%
j_{p}=i_{p}+v_{p}\text{.}  \label{eq-1-1-map}
\end{equation}%
For $\forall p\in \lbrack k]$, we define%
\begin{equation*}
a_{k,p}:=\dprod\limits_{s\geq p}^{k}y_{s}\dprod\limits_{s\geq p+1}^{k}x_{s},%
\text{ }b_{k,p}:=\dprod\limits_{1\leq s\leq p}^{{}}x_{s}y_{s}.
\end{equation*}%
With the one-to-one map (\ref{eq-1-1-map}), after some algebra we can see
that in $\widetilde{B}_{2k}$%
\begin{equation}
\sum_{D_{2k,l}}\dprod\limits_{p=1}^{k}x_{_{p}}^{i_{p}}y_{_{p}}^{j_{p}}=%
\left( \sum_{E_{k,l}}\dprod\limits_{p=1}^{k}a_{k,p}^{v_{p}}\right) \left(
\sum_{\widetilde{E}_{k,l}}\dprod\limits_{p=1}^{k}b_{k,p}^{\widetilde{v}%
_{p}}\right) .  \label{eq-after-map-1}
\end{equation}%
We claim that%
\begin{equation}
\sum_{E_{k,l}}\dprod\limits_{p=1}^{k}a_{k,p}^{v_{p}}=\sum_{p}\frac{%
a_{k,p}^{l}}{\dprod\limits_{s\neq p}(1-a_{k,s}/a_{k,p})},\text{ }\sum_{%
\widetilde{E}_{k,l}}\dprod\limits_{p=1}^{k}b_{k,p}^{\widetilde{v}%
_{p}}=\sum_{p}\frac{b_{k,p}^{l}}{\dprod\limits_{s\neq p}(b_{k,p}/b_{k,s}-1)}.
\label{eq-claim-geo-sum}
\end{equation}%
One way to prove (\ref{eq-claim-geo-sum}) is by induction on $l$ (also see
(2.9) and (2.10) in \cite{SFHP}. Letting $t=0$ in (2.10) and expanding the
generating function, we get (\ref{eq-claim-geo-sum})). Combining (\ref%
{eq-after-map-1}) and (\ref{eq-claim-geo-sum}), we get%
\begin{equation}
\sum_{D_{2k,l}}\dprod\limits_{p=1}^{k}x_{_{p}}^{i_{p}}y_{_{p}}^{j_{p}}=%
\sum_{1\leq p,q\leq k}\frac{\left( a_{k,p}b_{k,q}\right) ^{l}}{%
\dprod\limits_{s\neq p}(1-a_{k,s}/a_{k,p})\dprod\limits_{t\neq
q}(b_{k,q}/b_{k,t}-1)}.  \label{eq-posi-l-sum}
\end{equation}%
Next we consider $l\in \mathds{Z}_{-}$. Note that $[\phi
_{2k}(\sum_{D_{2k,l}}%
\prod_{p=1}^{k}y_{_{p-1}}^{i_{p}}x_{_{p}}^{j_{p}})]_{n}=\sum_{D_{2k,l}}%
\prod_{p=1}^{k}\overline{\alpha _{n+i_{p}}}\alpha _{n+j_{p}}$ where $%
y_{0}:=y_{k}$. Define%
\begin{equation*}
c_{k,p}:=\dprod\limits_{s\geq p}^{k}x_{s}\dprod\limits_{s\geq p}^{k-1}y_{s},%
\text{ }d_{k,p}:=\dprod\limits_{1\leq s\leq p}^{{}}y_{s-1}x_{s}.
\end{equation*}%
With the similar analysis to the $l>0$ case, we get that in $\widetilde{B}%
_{2k}$,%
\begin{equation*}
\sum_{D_{2k,l}}\dprod\limits_{p=1}^{k}y_{_{p-1}}^{i_{p}}x_{_{p}}^{j_{p}}=%
\left( \sum_{E_{k,l}}\dprod\limits_{p=1}^{k}c_{k,p}^{v_{p}}\right) \left(
\sum_{\widetilde{E}_{k,l}}\dprod\limits_{p=1}^{k}d_{k,p}^{\widetilde{v}%
_{p}}\right) .
\end{equation*}%
Let $e_{k,p}=c_{k,p}$ for $2\leq p\leq k$, and $e_{k,k+1}=e_{k,1}=c_{k,1}/%
\left( \prod_{i=1}^{k}x_{i}y_{i}\right) $. In $\widetilde{B}_{2k}$ we have 
\begin{equation*}
\left( \sum_{E_{k,l}}\dprod\limits_{p=1}^{k}c_{k,p}^{v_{p}}\right) \left(
\sum_{\widetilde{E}_{k,l}}\dprod\limits_{p=1}^{k}d_{k,p}^{\widetilde{v}%
_{p}}\right) =\left(
\sum_{E_{k,l}}\dprod\limits_{p=1}^{k}e_{k,p}^{v_{p}}\right) \left( \sum_{%
\widetilde{E}_{k,l}}\dprod\limits_{p=1}^{k}d_{k,p}^{\widetilde{v}%
_{p}}\right) .
\end{equation*}%
It is easy to verify that $\forall p,q\in \lbrack k]$,%
\begin{eqnarray}
d_{k,p}e_{k,q+1} &=&(\dprod\limits_{i=1}^{k}x_{i}y_{i})^{2}/(a_{k,p}b_{k,q})%
\text{,}  \notag \\
\dprod\limits_{s\neq p}(1-a_{k,s}/a_{k,p})\dprod\limits_{t\neq
q}(b_{k,q}/b_{k,t}-1) &=&\dprod\limits_{s\neq
q}(1-e_{k,s+1}/e_{k,q+1})\dprod\limits_{t\neq p}(d_{k,p}/d_{k,t}-1)\text{.}
\label{eq-relation-d-e-a-b}
\end{eqnarray}%
With (\ref{eq-claim-geo-sum}), (\ref{eq-relation-d-e-a-b}) and some algebra,
we get that in $\widetilde{B}_{2k}$%
\begin{equation}
\left( \sum_{E_{k,l}}\dprod\limits_{p=1}^{k}e_{k,p}^{v_{p}}\right) \left(
\sum_{\widetilde{E}_{k,l}}\dprod\limits_{p=1}^{k}d_{k,p}^{\widetilde{v}%
_{p}}\right) =\sum_{1\leq p,q\leq k}\frac{(a_{k,p}b_{k,q})^{-l}}{%
\dprod\limits_{s\neq p}(1-a_{k,s}/a_{k,p})\dprod\limits_{t\neq
q}(b_{k,q}/b_{k,t}-1)},  \label{eq-e-d-sum}
\end{equation}%
which is in $\widetilde{A}_{2k}$ by (\ref{eq-claim-geo-sum}). Now, combining
(\ref{def-V}), (\ref{eq-posi-l-sum}), (\ref{eq-e-d-sum}) and Lemma \ref%
{lemma-combinatorics}, we see that the following polynomial has the same
image as $G_{2k}$ under $\widetilde{\psi }_{2k}$:%
\begin{equation}
(-1)^{k+1}\frac{1}{kZ_{H}}\sum_{1\leq p,q\leq k}\frac{H(a_{k,p}b_{k,q})-h_{0}%
}{\dprod\limits_{s\neq p}(1-a_{k,s}/a_{k,p})\dprod\limits_{t\neq
q}(b_{k,q}/b_{k,t}-1)}.  \label{eq-same-image-G2k}
\end{equation}%
Finally we show where the term $-1/k$ comes from. According to the
definition of $Z_{H}$, it is easy to verify that $h_{0}=Z_{H}$. Because%
\begin{equation*}
\sum_{1\leq p,q\leq k}\frac{1}{\dprod\limits_{s\neq
p}(1-a_{k,s}/a_{k,p})\dprod\limits_{t\neq q}(b_{k,q}/b_{k,t}-1)}=\sum_{p}%
\frac{1}{\dprod\limits_{s\neq p}(1-a_{k,s}/a_{k,p})}\sum_{p}\frac{1}{%
\dprod\limits_{s\neq p}(b_{k,p}/b_{k,s}-1)}=(-1)^{k},
\end{equation*}%
which could be proved by induction, combined with (\ref{eq-same-image-G2k}), 
$-1/k$ appears and the proof is completed.
\end{proof}

\section{Proofs of the sum rules part\label{proof-sumrules-part}}

In this section we first show that for any $(\theta _{j},m_{j})_{1\leq j\leq
K}$, the degree 2 term $G_{2}$ matches the condition (\ref{condition-gz}).
This match recovers the result in \cite{GZ07}. We then provide the proof of
Theorem \ref{Theorem-n-0}, by the discrete Galiardo-Nirenberg Inequality and
a degree function $L_{2k}$ which relates Lukic's conditions to the algebra
model.

\begin{proof}[Proof of Theorem \protect\ref{theorem-GZ}]
Under the assumption $\alpha \in l^{4}$, we have $\limsup_{N\rightarrow
\infty }\sum_{n=0}^{N=1}(-\left\vert \alpha _{n}\right\vert ^{2}-\log
(1-\left\vert \alpha _{n}\right\vert ^{2}))<\infty $ (for example see
Proposition 4.1 in \cite{BSZ}). Note that in $\widetilde{B}_{2}$,%
\begin{equation*}
H\left( a_{1,1}b_{1,1}\right) =\frac{1}{2^{d}}\dprod%
\limits_{j=1}^{K}(x_{1}y_{1}^{2}-e^{i\theta
_{j}})^{m_{j}}(1/(x_{1}y_{1}^{2})-e^{-i\theta _{1}})^{m_{j}}=\frac{1}{2^{d}}%
\dprod\limits_{j=1}^{K}(y_{1}-e^{i\theta _{j}})^{m_{j}}(x_{1}-e^{-i\theta
_{1}})^{m_{j}}.
\end{equation*}%
Since%
\begin{equation}
\lbrack \phi _{2}(\frac{1}{2^{d}}\dprod\limits_{j=1}^{K}(y_{1}-e^{i\theta
_{j}})^{m_{j}}(x_{1}-e^{-i\theta _{1}})^{m_{j}})]_{n}=\left\vert
(\dprod\limits_{j=1}^{K}(S-e^{-i\theta _{j}})^{m_{j}}\alpha )_{n}\right\vert
^{2},  \label{eq-proof-gz-1}
\end{equation}%
applying Lemma \ref{lemma-alg-1} and Lemma \ref{lemma-alg-2} to $H\left(
a_{1,1}b_{1,1}\right) /Z_{H}$ and $\prod_{j=1}^{K}(y_{1}-e^{i\theta
_{j}})^{m_{j}}(x_{1}-e^{-i\theta _{1}})^{m_{j}}/2^{d}$, with (\ref%
{eq-proof-gz-1}) and Corollary \ref{corollary}, the proof is completed.
\end{proof}

The idea to prove Theorem \ref{Theorem-n-0} is the following. For each $k\in %
\mathds{Z}_{+}$, we define $L_{2k}$, a map from $A_{2k}$ to $\mathds{N}$, as
follows. For\thinspace $F\in A_{2k}$, we first do the Taylor expansion of $F$
at the point $(e^{-i\theta _{1}},e^{i\theta _{1}},e^{-i\theta
_{1}},...,e^{i\theta _{1}})$, such that%
\begin{equation*}
F=\sum_{s=1}^{M}C_{s}\dprod\limits_{p=1}^{k}(x_{p}-e^{-i\theta _{1}})^{\beta
_{p,s}}(y_{p}-e^{i\theta _{1}})^{\gamma _{p,s}}\text{,}
\end{equation*}%
where $M\in \mathds{Z}_{+}$, $\forall s\in \lbrack M],p\in \lbrack k],$ $%
C_{s}\neq 0,\beta _{p,s},\gamma _{p,s}\in $ $\mathds{N}$, and $(\beta
_{p,s},\gamma _{p,s})_{1\leq p\leq k}$ are distinct for different $s$. Then
let%
\begin{equation*}
L_{2k}(F):=\min_{1\leq s\leq M}\left( \sum_{p=1}^{k}\left( \beta
_{p,s}\wedge d+\gamma _{p,s}\wedge d\right) \right) .
\end{equation*}%
Since the Taylor expansion is unique, $L_{2k}$ is well-defined.

\begin{lemma}
\label{lemma-L}Under condition (\ref{condition-1}), if $F\in A_{2k}$ and $%
L_{2k}(F)\geq 2(d+1)-2k$, then%
\begin{equation*}
\limsup_{N\rightarrow \infty }\sum_{n=0}^{N}[\phi _{2k}(F)]_{n}<\infty .
\end{equation*}
\end{lemma}

To this end, in order to prove Theorem \ref{Theorem-n-0}, it suffices to
show the following lemma.

\begin{lemma}
\label{lemma-G-pp}$\forall k\in \lbrack d]$, there exists $G_{2k}^{\prime
\prime }\in A_{2k}$ such that $G_{2k}^{\prime \prime }=G_{2k}^{\prime }$ in $%
\widetilde{B}_{2k}$, and $L_{2k}(G_{2k}^{\prime \prime })\geq 2(d+1)-2k$.
\end{lemma}

To prove Lemma \ref{lemma-L}, we need the following discrete
Galiardo-Nirenberg Inequality. The references of this inequality are
Gagliardo \cite{Gag}, Nirenberg \cite{Niren}, and also see the the remark of
Theorem 2.5 in \cite{BSZ}, Section 6.3 of Simon \cite{ACCA} and Taylor \cite%
{Taylor}.

\begin{lemma}[Discrete Galiardo-Nirenberg Inequality]
\label{lemma-GN}If $(S-e^{-i\theta _{1}})^{d}\alpha \in l^{2}$ and $\alpha
\in l^{2d+2}$, then for any $j\in \lbrack d+1]$ we have%
\begin{equation*}
(S-e^{-i\theta _{1}})^{j}\alpha \in l^{\frac{2(d+1)}{j+1}}.
\end{equation*}
\end{lemma}

\begin{remark}
Since in our case $\left\Vert \alpha \right\Vert _{\infty }\leq 1$, under
the same conditions we have $(S-e^{-i\theta _{1}})^{q}\alpha \in l^{\frac{%
2(d+1)}{j+1}}$ for any $q\geq j$.
\end{remark}

Now we prove Lemma \ref{lemma-L}.

\begin{proof}[Proof of Lemma \protect\ref{lemma-L}]
It suffices to prove it for $F=\prod_{p=1}^{k}(x_{p}-e^{-i\theta
_{1}})^{\beta _{p}}(y_{p}-e^{i\theta _{1}})^{\gamma _{p}}$ with%
\begin{equation*}
\sum_{p=1}^{k}\left( \beta _{p}\wedge d+\gamma _{p}\wedge d\right) \geq
2(d+1)-2k.
\end{equation*}%
Write $\widetilde{\beta }_{p}=\beta _{p}\wedge d$ and $\widetilde{\gamma }%
_{p}=\gamma _{p}\wedge d$. Let $\lambda :=\sum_{p=1}^{k}\left( \widetilde{%
\beta }_{p}+1+\widetilde{\gamma }_{p}+1\right) $, then with the H\H{o}lder's
Inequality we have{\small 
\begin{equation}
\sum_{n}\dprod\limits_{p=1}^{k}\left\vert \left( (S-e^{-i\theta
_{1}})^{\beta _{p}}\alpha \right) _{n}\right\vert \left\vert \left(
(S-e^{-i\theta _{1}})^{\gamma _{p}}\alpha \right) _{n}\right\vert \leq
\dprod\limits_{p=1}^{k}||\left( (S-e^{-i\theta _{1}})^{\beta _{p}}\alpha
\right) ||_{\frac{\lambda }{\widetilde{\beta }_{p}+1}}||\left(
(S-e^{-i\theta _{1}})^{\gamma _{p}}\alpha \right) ||_{\frac{\lambda }{%
\widetilde{\gamma }_{p}+1}}.  \label{eq-proof-lemma-L}
\end{equation}%
}Recalling that $[\phi _{2k}(F)]_{n}=\prod_{p=1}^{k}((S-e^{-i\theta
_{1}})^{\beta _{p}}\alpha )_{n}\overline{((S-e^{-i\theta _{1}})^{\gamma
_{p}}\alpha )_{n}}$, with the fact that $\left\Vert \alpha \right\Vert
_{\infty }\leq 1$ and $\lambda \geq 2(d+1)$, we finish the proof by (\ref%
{eq-proof-lemma-L}) and Lemma \ref{lemma-GN}.
\end{proof}

Next we show Lemma \ref{lemma-G-pp}. First we provide a method to calculate
the polynomial $G_{2k}^{\prime }$. Define an operator $D(x_{1},...,x_{n})(%
\cdot )$ from $\mathbb{C}[x,1/x]$ to $\mathbb{C}%
[x_{1},...,x_{n},1/x_{1},...,1/x_{n}]$ as follows: for $f(x)=%
\sum_{i=-d_{1}}^{d_{2}}c_{i}x^{i}$, let%
\begin{equation*}
D(x_{1},...,x_{n})(f):=\sum_{i=1}^{n}\frac{f(x_{i})}{\Pi _{j\neq
i}(x_{j}-x_{i})}\text{.}
\end{equation*}%
We can observe that $D(x_{1},...,x_{n})(f)$ is a Hall-Littlewood type
polynomial, and%
\begin{equation}
D(x_{1},...,x_{n})(f)=\frac{%
D(x_{1},x_{3},...,x_{n})(f)-D(x_{2},x_{3},...,x_{n})(f)}{x_{2}-x_{1}}.
\label{eq-D-iterative}
\end{equation}

\begin{proof}[Proof of Lemma \protect\ref{lemma-G-pp}]
Let%
\begin{equation*}
f_{1}(b_{k,q},x):=\left( xb_{k,q}-e^{i\theta _{1}}\right) ^{d}\left( (\Pi
_{i}x_{i}y_{i})^{2}/(xb_{k,q})-e^{-i\theta _{1}}\right)
^{d}x^{k-1}b_{k,q}^{-1},
\end{equation*}%
\begin{equation*}
f_{2}(a_{k,1},...,a_{k,k},x):=D(a_{k,1},...,a_{k,k})(f(x,\cdot )).
\end{equation*}%
After some algebra we can see that%
\begin{equation}
\sum_{1\leq p,q\leq k}\frac{\left( a_{k,p}b_{k,q}-e^{i\theta _{1}}\right)
^{d}\left( (\Pi _{i}x_{i}y_{i})^{2}/(a_{k,p}b_{k,q})-e^{-i\theta
_{1}}\right) ^{d}}{\dprod\limits_{s\neq
p}(1-a_{k,s}/a_{k,p})\dprod\limits_{t\neq q}(b_{k,q}/b_{k,t}-1)}%
=D(b_{k,1},...,b_{k,k})(f_{2}(a_{k,1},...,a_{k,k},\cdot
))\dprod\limits_{t=1}^{k}b_{k,t}.  \label{eq-final-1}
\end{equation}%
Note that for any $r_{1},...,r_{M}\in \mathbb{C}$ and $\beta \in \mathbb{Z}$%
, we have%
\begin{equation}
\frac{x_{1}^{\beta }\dprod\limits_{i=1}^{M}(x_{1}-r_{i})-x_{2}^{\beta
}\dprod\limits_{i=1}^{M}(x_{2}-r_{i})}{x_{1}-x_{2}}=(\sum_{j=0}^{\beta
-1}x_{1}^{j}x_{2}^{\beta
-1-j})\dprod\limits_{i=1}^{M}(x_{1}-r_{i})+x_{2}^{\beta
}(\sum_{i=1}^{M}\dprod\limits_{s<i}(x_{1}-r_{s})\dprod%
\limits_{t>i}(x_{2}-r_{t})),  \label{eq-expand}
\end{equation}%
where in the right hand side each term contains a factor in the form of%
\begin{equation}
\Pi _{s<i}(x_{1}-r_{s})\Pi _{t>i}(x_{2}-r_{t}),\text{with }i\in
\{0,1,2,...,M\}.  \label{eq-cut-argument}
\end{equation}%
Write{\small 
\begin{eqnarray}
&&\left( xb_{k,q}-e^{i\theta _{1}}\right) ^{d}\left( \frac{(\Pi
_{i}x_{i}y_{i})^{2}}{xb_{k,q}}-e^{-i\theta _{1}}\right) ^{d}  \notag \\
&=&\left( xb_{k,q}-e^{i\theta _{1}}\right) \left( \frac{(\Pi
_{i}x_{i}y_{i})^{2}}{xb_{k,q}}-e^{-i\theta _{1}}\right) ...\left(
xb_{k,q}-e^{i\theta _{1}}\right) \left( \frac{(\Pi _{i}x_{i}y_{i})^{2}}{%
xb_{k,q}}-e^{-i\theta _{1}}\right) ,  \label{eq-rewrite}
\end{eqnarray}%
}where the right hand side is a multiplication of $2d$ terms. With (\ref%
{eq-D-iterative}), (\ref{eq-expand}) and (\ref{eq-rewrite}), it is not hard
to observe that, we can express (\ref{eq-final-1}) as a summation, where
each term in this summation contains a factor like%
\begin{equation}
\dprod\limits_{p,q}\left( a_{k,p}b_{k,q}-e^{i\theta _{1}}\right)
^{d_{p,q}^{(1)}}\left( \frac{(\Pi _{i}x_{i}y_{i})^{2}}{xb_{k,q}}-e^{-i\theta
_{1}}\right) ^{d_{p,q}^{(2)}}=\dprod\limits_{p,q}\left(
a_{k,p}b_{k,q}-e^{i\theta _{1}}\right) ^{d_{p,q}^{(1)}}\left(
d_{k,p}e_{q+1}-e^{-i\theta _{1}}\right) ^{d_{p,q}^{(2)}}.
\label{eq-contain-factor}
\end{equation}%
What's more, each $(d_{p,q}^{(1)},d_{p,q}^{(2)})_{p,q\in \lbrack k]}$
corresponds to some $(\widetilde{d}_{p,q}^{(1)},\widetilde{d}%
_{p,q}^{(2)})_{p,q\in \lbrack k]}$ generated as follows. Put $d$ white balls
and $d$ black balls alternately, that is, White, Black, ..., White, Black.
Here White stands for $\left( xb_{k,q}-e^{i\theta _{1}}\right) $, and Black
stands for $\left( (\Pi _{i}x_{i}y_{i})^{2}/(xb_{k,q})-e^{-i\theta
_{1}}\right) $. Choose $0=z_{0}\leq z_{1}\leq z_{2}\leq ...\leq z_{k-1}\leq
z_{k}=2d$ where $z_{i}\in \{0,1,...,2d\}$ $\forall i\in \lbrack k-1]$, and $%
0=w_{0}\leq w_{1}\leq w_{2}\leq ...\leq w_{k-1}\leq w_{k}=2d$ where $%
w_{i}\in \{0\}\cup \{\{0,1,...,d\}\backslash \{z_{1},...,z_{k-1}\}\}$ $%
\forall i\in \lbrack k-1]$, then let%
\begin{eqnarray*}
\widetilde{d}_{p,q}^{(1)} &:&=\#\{\text{White balls in }[z_{p-1},z_{p}]\cap
\lbrack w_{q-1},w_{q}]\}, \\
\widetilde{d}_{p,q}^{(2)} &:&=\#\{\text{Black balls in }[z_{p-1},z_{p}]\cap
\lbrack w_{q-1},w_{q}]\}.
\end{eqnarray*}%
The correspondence between $(d_{p,q}^{(1)},d_{p,q}^{(2)})_{p,q\in \lbrack k]}
$ and $(\widetilde{d}_{p,q}^{(1)},\widetilde{d}_{p,q}^{(2)})_{p,q\in \lbrack
k]}$ could be observed by (\ref{eq-cut-argument}). Therefore, we see that $%
d_{p,q}^{(1)},d_{p,q}^{(2)}\in \mathbb{[}d\mathbb{]}$, $\left\vert
d_{p,q}^{(1)}-d_{p,q}^{(2)}\right\vert \leq 1$, and%
\begin{equation}
\sum_{p,q}(d_{p,q}^{(1)}+d_{p,q}^{(2)})\geq 2d-2(k-1).  \label{eq-sum-d-p-q}
\end{equation}%
where the last inequality holds because in order to get (\ref{eq-final-1}),
we need to do the operation like (\ref{eq-D-iterative}) for $2(k-1)$ times,
and each operation at most reduce degree $1$ for these factors. Note that
there are at most $2k-1$ pairs $(p,q)$ with $[z_{p-1},z_{p}]\cap \lbrack
w_{q-1},w_{q}]\neq \emptyset $, thus $\sum |d_{p,q}^{(1)}-d_{p,q}^{(2)}|\leq
2k-1$, and with (\ref{eq-sum-d-p-q}) we see that%
\begin{equation}
\max_{i_{p,q}\in \{1,2\}\text{ }\forall p,q\in \lbrack
k]}\sum_{p,q}d_{p,q}^{(i_{p,q})}\leq d\text{.}  \label{eq-max-d-p-q-d}
\end{equation}%
Now consider any $f_{3},f_{4}\in A_{2k}$ where%
\begin{equation*}
f_{3}=\Pi _{p,q}\left( a_{k,p}b_{k,q}-e^{i\theta _{1}}\right)
^{d_{p,q}^{(1)}}\left( d_{k,p}e_{q+1}-e^{-i\theta _{1}}\right)
^{d_{p,q}^{(2)}}f_{4}.
\end{equation*}%
By the definition of $a_{k,p}b_{k,q}$ and $d_{k,p}e_{q+1}$, we can see that $%
f_{3}$ has the same image with the following polynomial $\widetilde{f}_{3}$
under $\phi _{2k}$:%
\begin{equation*}
\widetilde{f}_{3}:=\Pi _{p,q}\left( \Pi _{s\in I_{p,q}^{(1)},t\in
I_{p,q}^{(2)}}x_{s}y_{t}-e^{i\theta _{1}}\right) ^{d_{p,q}^{(1)}}\left( \Pi
_{s\in \lbrack k]\backslash I_{p,q}^{(1)},t\in \lbrack k]\backslash
I_{p,q}^{(2)}}x_{s}y_{t}-e^{-i\theta _{1}}\right) ^{d_{p,q}^{(2)}}f_{4}
\end{equation*}%
with $I_{p,q}^{(1)},I_{p,q}^{(2)}\subset \lbrack k]$ and $%
|I_{p,q}^{(1)}|=|I_{p,q}^{(2)}|-1$. Expand each $\Pi _{s\in
I_{p,q}^{(1)},t\in I_{p,q}^{(2)}}x_{s}y_{t}-e^{i\theta _{1}}$ as a Taylor
series, whose degree $1$ terms are exactly $\sum_{s\in I_{p,q}^{(1)}}\left(
x_{s}-e^{-i\theta _{1}}\right) +\sum_{t\in I_{p,q}^{(2)}}\left(
y_{t}-e^{i\theta _{1}}\right) $, and each higher degree term is divided by
some degree $1$ term. We claim that if we apply $L_{2k}$ on each lowest
degree term in the Taylor expansion of $\widetilde{f}_{3}/f_{4}$, the result 
$\geq 2d-2(k-1)$, since for each such term, the degree of any $\left(
x_{s}-e^{-i\theta _{1}}\right) $ and $\left( y_{t}-e^{i\theta _{1}}\right) $
is $\leq d$ by (\ref{eq-max-d-p-q-d}), and the total degree is $\geq
2d-2(k-1)$ by (\ref{eq-sum-d-p-q}). Noting that each term in the Taylor
expansion of $\widetilde{f}_{3}/f_{4}$ is divided by some lowest degree term
in the expansion, we see that $L_{2k}(\widetilde{f}_{3})\geq 2d-2(k-1)$, and
the proof is completed by (\ref{eq-contain-factor}).
\end{proof}

\noindent \textbf{Acknowledgements} The author thank Ofer Zeitouni for introducing
this problem to him, and Amir Dembo, Ofer Zeitouni for many helpful
discussions and comments.

{\small 
\bibliographystyle{plain}
\bibliography{LC}
}

\bigskip

\bigskip

\bigskip

\bigskip

\bigskip

\end{document}